\documentclass{amsart}

\usepackage{begnac}

\begin{document}

\title{Lipschitz functions on topometric spaces}

\author{Itaï \textsc{Ben Yaacov}}

\address{Itaï \textsc{Ben Yaacov} \\
  Université Claude Bernard -- Lyon 1 \\
  Institut Camille Jordan, CNRS UMR 5208 \\
  43 boulevard du 11 novembre 1918 \\
  69622 Villeurbanne Cedex \\
  France}

\urladdr{\url{http://math.univ-lyon1.fr/~begnac/}}

\thanks{Research supported by the Institut Universitaire de France and ANR contract GruPoLoCo (ANR-11-JS01-008).}

\svnInfo $Id: LipTM.tex 1497 2013-01-29 14:44:17Z begnac $
\thanks{\textit{Revision} {\svnInfoRevision} \textit{of} \today}

%\date{\today}

\begin{abstract}
  We study functions on topometric spaces which are both (metrically) Lipschitz and (topologically) continuous, using them in contexts where, in classical topology, ordinary continuous functions are used.
  We study the relations of such functions with topometric versions of classical separation axioms, namely, normality and complete regularity, as well as with completions of topometric spaces.
  We also recover a compact topometric space $X$ from the lattice of continuous $1$-Lipschitz functions on $X$, in analogy with the recovery of a compact topological space $X$ from the structure of (real or complex) functions on $X$.
\end{abstract}

\maketitle

\tableofcontents

\section*{Introduction}

Compact topometric spaces were first defined in \cite{BenYaacov-Usvyatsov:CFO} as a formalism for various global and local type spaces arising in the context of continuous first order logic, allowing for some kind of (topometric) Cantor-Bendixson analysis in spaces which, from a purely topological point of view, are possibly even perfect.
General topometric spaces (i.e., non compact) were defined studied further from an abstract point of view in \cite{BenYaacov:TopometricSpacesAndPerturbations}, where the formalism is shown to be useful for the analysis of perturbation structures on type spaces.
The same idea was also shown to be useful in the context of (very non compact) Polish groups, which may admit ``topometric ample generics'' even when no purely topological ample generics need exist, see \cite{BenYaacov-Berenstein-Melleray:TopometricGroups}.

In a nutshell, topometric spaces are spaces equipped both with a metric and a topology, \emph{which need not agree}.

\begin{dfn}
  \label{dfn:TopoMetric}
  A \emph{topometric space} is a triplet $(X,\sT,d)$, where $\sT$ is a topology and $d$ a metric on $X$, satisfying:
  \begin{enumerate}
  \item The distance function $d\colon X^2 \to [0,\infty]$ is lower semi-continuous in the topology.
  \item The metric refines the topology.
  \end{enumerate}
\end{dfn}

We follow the convention that unless explicitly qualified, the vocabulary of general topology (compact, continuous, etc.)\ refers to the topological structure, while the vocabulary of metric spaces (Lipschitz function, etc.)\ refers to the metric structure.
Excluded from this convention are separation axioms: we assimilate the lower semi-continuity of the distance function to the Hausdorff separation axiom, and stronger axioms, such as normality and complete regularity, will be defined for topometric spaces below.

We shall refer to two classes of examples, arising from the embedding of the categories of (Hausdorff) topological spaces and of metric spaces in the category of topometric spaces.
By a \emph{maximal} topometric space we mean one equipped with the discrete $0/\alpha$ distance for some $0 < \alpha \leq \infty$.
Such a space can be identified, for most intents and purposes, with its underlying pure topological structure.
In particular, every bounded function is Lipschitz, and if $\alpha = \infty$ then every function is $k$-Lipschitz for any $k$.
Similarly, a \emph{minimal} topometric space is one in which the metric and topology agree, which may be identified with its underlying metric structure.
These mostly serve as first sanity checks (e.g., when we define a normal topometric space we must check that a maximal one is normal if and only if it is normal as a pure topological space, and that minimal ones are always normal).

In the same way that much information can be gained on a topological space from spaces of continuous functions on such a space, we seek here to gain information on a topometric space from the (topologically) continuous and (metrically) Lipschitz functions thereon.
These are naturally linked with separation axioms.
In \fref{sec:Normal} we discuss topometric normality, which we related to existence results such as Urysohn's Lemma and Tietze's Extension Theorem.
As a consequence, we obtain a Lipschitz Morleyisation result, the unique model-theoretic result of this paper.
In \fref{sec:StoneCech} we construct the topometric Stone-Čech compactification and relate it to topometric complete regularity.
In \fref{sec:Completions} we study possible topological structures on the metric completion of a topometric space, relating these (to an extent) with completely regular spaces.
In \fref{sec:Abstract1Lipschitz} we characterise the spaces of continuous $1$-Lipschitz functions on compact topometric space, and show that the original topometric space can be recovered uniquely from its function space.

The reader is advised that Lipschitz functions on an ordinary metric spaces, and algebras thereof, are extensively studied by Weaver \cite{Weaver:LipschitzAlgebras}.
There is some natural resemblance between our object of study here and that of Weaver, with the increased complexity due to the additional topological structure.
The reader may wish to compare, for example, our version of Tietze's Extension Theorem (\fref{thm:Tietze}) with \cite[Theorem~1.5.6]{Weaver:LipschitzAlgebras} (as well as with the classical version of Tietze's Theorem, see Munkres \cite{Munkres:Topology}).

\section{Normal topometric spaces and Urysohn and Tietze style results}
\label{sec:Normal}

For two topometric spaces $X$ and $Y$ we define $\cC_{\cL(1)}(X,Y)$ to be the set of all continuous $1$-Lipschitz functions from $X$ to $Y$.
An important special case is $\cC_{\cL(1)}(X) = \cC_{\cL(1)}(X,\bC)$, where $\bC$ is equipped with the standard metric and topology (i.e., with the standard minimal topometric structure), which codes information both about the topology and about the metric structure of $X$.
In the present paper we seek conditions under which $\cC_{\cL(1)}(X)$ codes the entire topometric structure, as well as analogues of classical results related to separation axioms, in which $\cC(X)$ would be replaced with $\cC_{\cL(1)}(X)$.
As discussed in \cite{BenYaacov:TopometricSpacesAndPerturbations}, we consider the lower semi-continuity of the distance function to be a topometric version of the Hausdorff separation axiom, so it is expected that other classical separation axioms take a different form in the topometric setting.
We start with normality.

\begin{dfn}
  Let $X$ be a topometric space.
  We say that a closed set $F \subseteq X$ has \emph{closed metric neighbourhoods} if for every $r > 0$ the set $\overline B(F,r) = \{x \in X : d(x,F) \leq r\}$ is closed in $X$.

  We say that $X$ \emph{admits closed metric neighbourhoods} if all closed subsets of $X$ do.
\end{dfn}

It was shown in \cite{BenYaacov:TopometricSpacesAndPerturbations} that compact sets always have closed metric neighbourhoods, so a compact topometric space admits closed metric neighbourhoods.
Indeed, the first definition of a \emph{compact} topometric space in \cite{BenYaacov-Usvyatsov:CFO} was given in terms of closed metric neighbourhoods.
While this property seems too strong to be part of the definition of a non compact topometric space, it will play a crucial role in this section.

\begin{dfn}
  A \emph{normal topometric space} is a topometric space $X$ satisfying:
  \begin{enumerate}
  \item Every two closed subset $F,G \subseteq X$ with positive distance $d(F,G) > 0$ can be separated by disjoint open sets.
  \item The space $X$ admits closed metric neighbourhoods.
  \end{enumerate}
\end{dfn}

One checks that a maximal topometric space $X$ (i.e., equipped with the discrete $0/\alpha$ distance) is normal if and only if it is so as a topological space.
Similarly, a minimal topometric space (i.e., equipped with the metric topology) is always normal.
Also, every compact topometric space is normal (since it admits closed metric neighbourhoods and the underlying topological space is normal).

We contend that our definition of a normal topometric space is the correct topometric analogue of the classical notion of a normal topological space.
This will be supported by analogues of Urysohn's Lemma and of Tietze's Extension Theorem.
The technical core of the proofs (and indeed, the only place where the definition of a normal topometric space is used) lies in the following Definition and Lemma.

\begin{dfn}
  Let $X$ be a topometric space.
  \begin{enumerate}
  \item By an \emph{($S$-)system} we mean a sequence $\Xi = (F_\alpha,G_\alpha)_{\alpha \in S} = (F^\Xi_\alpha,G^\Xi_\alpha)_{\alpha \in S^\Xi}$ where $S \subseteq \bR$ and $F_\alpha,G_\alpha \subseteq X$ are closed, and if $\alpha < \beta$ then $F_\alpha \cap G_\beta = \emptyset$.
    We say that it is \emph{$c$-Lipschitz} for some $c > 0$ if $d(F_\alpha,G_\beta) \geq (\beta - \alpha)/c$ for $\alpha < \beta$ in $S$.
    It is \emph{total} if $F_\alpha \cup G_\alpha = X$ for all $\alpha \in S$.
  \item We say that $\Xi \subseteq \Xi'$ if $S^\Xi \subseteq S^{\Xi'}$ and $F^\Xi_\alpha \subseteq F^{\Xi'}_\alpha$, $G^\Xi_\alpha \subseteq G^{\Xi'}_\alpha$ for all $\alpha \in S^\Xi$.
  \item To each continuous function $f\colon X \rightarrow \bR$ we associate a total system $\Xi^f = (F^f_\alpha,G^f_\alpha)_{\alpha \in \bR}$ where $F^f_\alpha = \{x : f(x) \leq \alpha\}$, $G^f_\alpha = \{x : f(x) \geq \alpha\}$.
  \item We say that $\Xi = (F_\alpha,G_\alpha)_{\alpha \in S}$ is \emph{compatible} with $f\colon X \to \bR$ if $f\rest_{F_\alpha} \leq \alpha$ and $f\rest_{G_\alpha} \geq \alpha$ for $\alpha \in S$, i.e., if $\Xi \subseteq \Xi^f$.
  \end{enumerate}
\end{dfn}

Notice that a continuous function $f\colon X\to \bR$ is $c$-Lipschitz if and only if $\Xi^f$ is a $c$-Lipschitz system.

\begin{lem}
  \label{lem:LipschitzSystemExtension}
  Let $\Xi$ be a finite $c$-Lipschitz $S$-system in a normal topometric space.
  Then for every $c' > c$ and countable $S' \supseteq S$ there exists a total $c'$-Lipschitz $S'$-system $\Xi' \supseteq \Xi$.
\end{lem}
\begin{proof}
  Assume first that $S'$ is finite, and by adding $F_\alpha = G_\alpha = \emptyset$ for $\alpha \in S' \setminus S$ we may assume that $S' = S$.
  It is then enough to find $F'_\beta \supseteq F_\beta$ and $G'_\beta \supseteq G_\beta$ for one $\beta \in S$, such that $F'_\beta \cup G'_\beta = X$, keeping the other $F_\alpha$, $G_\alpha$ unchanged.
  Since the partial approximation is finite it is also $c'$-Lipschitz for some $c' < c$.
  Define:
  \begin{gather*}
    K = \bigcup_{\alpha \in S, \alpha < \beta} \overline B(F_\alpha,(\beta-\alpha)/c'),
    \qquad
    L = \bigcup_{\alpha \in S, \alpha > \beta} \overline B(G_\alpha,(\alpha-\beta)/c').
  \end{gather*}
  By construction $d(K,L) > 0$ and both are closed as finite unions of closed sets.
  Since $X$ is normal we can find disjoint open sets $U \supseteq K$ and $V \supseteq L$.

  We claim that $F'_\beta = F_\beta \cup V^c$ and $G'_\beta = G_\beta \cup U^c$ will do.
  Indeed, $X = F'_\beta \cup G'_\beta$.
  Assume now that $\alpha < \beta$, $\alpha \in S$.
  We already know by hypothesis that $d(F_\alpha,G_\beta) > \frac{\beta - \alpha}{c'}$.
  We also know by construction that $U \supseteq \overline B(F_\alpha,(\beta-\alpha)/c')$, whereby $d(F_\alpha,G_\beta') \geq \frac{\beta - \alpha}{c'}$.
  Similarly, if $\beta < \alpha$ then $d(F'_\beta,G_\alpha) \geq \frac{\alpha - \beta}{c'}$, and we are done.

  The case where $S'$ is infinite countable follows by induction, adding a single new element at a time.
\end{proof}

\begin{lem}
  \label{lem:LipschitzSystemFunction}
  In a normal topometric space $X$, for every finite $c$-Lipschitz system $\Xi$ and $c' > c$ there exists a continuous $c'$-Lipschitz $f\colon X \rightarrow \bR$ compatible with $\Xi$.
\end{lem}
\begin{proof}
  Let $\Xi = (F_\alpha,G_\alpha)_{\alpha \in S}$.
  Since $S$ is finite its convex hull is a compact interval $I \subseteq \bR$.
  Let $T \subseteq I$ be a countable dense subset containing $S$.
  By \fref{lem:LipschitzSystemExtension}, there exists a total $c'$-Lipschitz $T$-system $\Xi' = (F'_\alpha,G'_\alpha)_{\sigma \in T} \supseteq \Xi$.
  Letting $f(x) = \sup\{\alpha\in I : x \in G'_\alpha\} = \inf\{\alpha \in I : x\in F'_\alpha\}$ (here $\inf \emptyset = \sup I$ and $\sup \emptyset = \inf I$) one obtains a continuous, $c'$-Lipschitz function $f\colon X \to I$ compatible with $\Xi$.
\end{proof}

\begin{thm}[Urysohn's Lemma for topometric spaces]
  \label{thm:TopometricUrysohn}
  Let $X$ be a normal topometric space, $F,G \subseteq X$ closed sets, $0 < r < d(F,G)$.
  Then there exists a $1$-Lipschitz continuous function $f\colon X\to [0,r]$ equal to $0$ on $F$ and to $r$ on $G$.
\end{thm}
\begin{proof}
  Apply \fref{lem:LipschitzSystemFunction} to $S = \{0,r\}$, $F_0 = F$, $G_r = G$, $G_0 = F_r = X$.
\end{proof}

\begin{cor}
  \label{cor:DenseLipschitz}
  Let $X$ be a compact topometric space.
  Then the family of continuous Lipschitz functions on $X$ is dense in $\cC(X)$.
\end{cor}
\begin{proof}
  It will be enough to show that the family $\cC_\cL(X,\bR)$ of real-valued continuous Lipschitz functions on $X$ is uniformly dense in $\cC(X,\bR)$.
  Since $X$ is compact, it is normal.
  The family $\cC_\cL(X,\bR)$ forms a lattice, and in addition, for every two distinct points $x,y \in \tS_n(\cL)$ and values $r,s \in \bR$, there exists by Urysohn's Lemma $f \in \cC_\cL(X,\bR)$ such that $f(x) = s$ and $f(y) = r$.
  By the lattice version of the Stone-Weierstraß theorem, $\cC_\cL(X,\bR)$ is uniformly dense in $\cC(X,\bR)$.
\end{proof}

We turn to prove a topometric version of Tietze's Extension Theorem, which, as for ordinary topological spaces, can be viewed as a strengthening of Urysohn's Lemma.

\begin{fct}
  There exists a function $\flim\colon \bR^\bN \rightarrow \bR$ which is continuous in the product topology and $1$-Lipschitz in the uniform topology, such that if $|t_n - s| \leq 2^{-n}$ for all $n$ then $\flim(t_n) = s$.
\end{fct}
\begin{proof}
  See \cite{BenYaacov-Usvyatsov:CFO}.
\end{proof}

\begin{thm}[Tietze's Extension Theorem for topometric spaces]
  \label{thm:Tietze}
  Let $X$ be a normal topometric space.
  Then for every $c < c'$ every continuous $c$-Lipschitz function $f\colon Y \to [0,1]$ on a closed subset $Y \subseteq X$ extends to a continuous $c'$-Lipschitz function $g\colon X \to [0,1]$.

  Moreover, for an arbitrary topometric space the following are equivalent:
  \begin{enumerate}
  \item $X$ is a normal topometric space.
  \item Tietze's Extension Theorem for topometric spaces (i.e., the statement above) holds in $X$.
  \item Urysohn's Lemma (the conclusion of \fref{thm:TopometricUrysohn}) holds in $X$.
  \end{enumerate}
\end{thm}
\begin{proof}
  Let $Y \subseteq X$ be closed, $f\colon Y \to [0,1]$ be continuous and $c$-Lipschitz.

  For $n \in \bN$ let $S_n = \{k2^{-n} : 0 \leq k \leq 2^n\}$, and $\Xi_n = (F^f_\alpha,G^f_\alpha)_{\alpha \in S_n}$.
  Then $\Xi_n$ is a $c$-Lipschitz system.
  By \fref{lem:LipschitzSystemFunction} there is a $c'$-Lipschitz function $g_n\colon X \rightarrow [0,1]$ compatible with $\Xi_n$.
  In particular, if $y \in Y$ and $k2^{-n} \leq f(y) \leq (k+1)2^{-n}$ then $y \in F^f_{(k+1)2^{-n}} \cap G^f_{k2^{-n}} \subseteq F^{g_n}_{n,(k+1)2^{-n}} \cap G^{g_n}_{n,k2^{-n}}$, whereby $k2^{-n} \leq g_n(y) \leq (k+1)2^{-n}$ as well.

  Let $g = \flim g_n$.
  Then $g$ is continuous and $c'$-Lipschitz, and $|g_n\rest_Y-f| \leq 2^{-n}$ for all $n$ implies $g\rest_Y = f$.

  For the moreover part:
  \begin{cycprf}
  \item This is the main assertion.
  \item Immediate.
  \item[\impfirst]
    Assume Urysohn's Lemma holds in $X$.
    Then closed sets of positive distance can be separated by a $1$-Lipschitz continuous function, and therefore by open sets.
    Also, if $F \subseteq X$ is closed and $d(x,F) > r$ then we may separate $F$ and $x$ by a $1$-Lipschitz continuous function such that $f\rest_F = 0$ and $f(x) > r$.
    Then $\{y : f(y) \leq r\}$ is a closed set containing $\overline B(F,r)$ but not $x$.
    If follows that $\overline B(F,r)$ is closed.
  \end{cycprf}
\end{proof}

This proof of Tiezte's theorem is fairly different from proofs we found in the literature, and which do not seem to be capable of preserving the Lipschitz condition.

For the last result of this section we shall assume some familiarity with continuous logic.
A language for continuous logic was defined in \cite{BenYaacov-Usvyatsov:CFO} to consist of a collection of symbols equipped with uniform continuity moduli, which their interpretations are required to respect.
Arbitrary continuity moduli  are allowed since, first, this extra generality creates no additional difficulties, and second, even if we had required all symbols to be, say, $1$-Lipschitz, arbitrary definable predicates would still merely be uniformly continuous, creating an inconvenient discrepancy.
That said, we can now show that in many situations one may assume that the language is indeed $1$-Lipschitz.

\begin{thm}[Lipschitz Morleyisation]
  \label{thm:LipschitzMorleyisation}
  Let $\cL$ be any continuous language.
  \begin{enumerate}
  \item
    The family of Lipschitz $\cL$-definable predicates is uniformly dense in the family of all definable predicates and witness distances between types.
  \item
    There exists a $1$-Lipschitz relational language $\cL'$ of cardinality $|\cL| + \aleph_0$ and an $\cL'$-theory $T_0$, such that the class of $\cL$-structures stands in a bidefinable bijection with the class of models of $T_0$, and moreover, for any two $p,q \in \tS_n(T_0)$ we have $d(p,q) = \sup_P |P^p - P^q|$ as $P$ varies over all $n$-ary predicate symbols in $\cL'$ (so in particular, $T_0$ eliminates quantifiers).
    This bijection necessarily respects elementary embeddings, ultra-products, elementary sub-classes, and so on.
  \end{enumerate}
\end{thm}
\begin{proof}
  For the first assertion, for each $n$ the space $\tS_n(\cL)$ of complete $n$-types in $\cL$ is compact, so we may apply \fref{cor:DenseLipschitz}, observing that the $n$-ary $\cL$-definable predicates are in a natural bijection with the continuous function on $\tS_n(\cL)$, and that this bijection respects uniform distance and uniform continuity moduli.
  The second assertion follows.
\end{proof}

\section{Completely regular topometric spaces and Stone-Čech compactification}
\label{sec:StoneCech}

Let $\{X_i : i \in I\}$ be a family of topometric spaces.
We equip the set $\prod_{i \in I} X_i$ with the product topology and the supremum metric $d(\bar x,\bar y) = \sup\{d(x_i,y_i) : i \in I\}$.
One verifies easily the result is indeed a topometric space which we call the \emph{product topometric structure}.

In particular we obtain large compact topometric spaces of the form $[0,\infty]^I$, and we claim that these are in some sense universal, meaning that every compact topometric space embeds in one of those.
Similarly, every bounded compact topometric (i.e., of finite diameter) can be embedded in $[0,M]^I$, and up to re-scaling in $[0,1]^I$.
In fact we shall show that every completely regular topometric space embeds in such a space, obtaining a Stone-Čech compactification.

\begin{dfn}
  \label{dfn:Compatification}
  By an \emph{embedding} of topometric spaces we mean a map $X \hookrightarrow Y$ between topometric spaces which is both a topological (homeomorphic) embedding and an isometric one.

  By a \emph{compactification} of a topometric space $X$ we mean a topometric embedding in a compact topometric space with dense image.
\end{dfn}

Say that a family of functions $\cF \subseteq \bC^X$ \emph{separates points from closed sets} if for every closed set $F \subseteq X$ and $x \in X \setminus F$, there is a function $f \in \cF$ which is constant on $F$ and takes some different value at $x$.

\begin{fct}
  \label{fct:TopologicalCompletelyRegularEmbedding}
  Let $X$ be a Hausdorff topological space, $\cF \subseteq \cC(X)$ a family separating points from closed sets.
  Then the map $\theta\colon X \to \bC^\cF$ defined by $x \mapsto (f \mapsto f(x))$ is a topological embedding.
\end{fct}
\begin{proof}
  Standard.
  % This is fairly standard.
  % First of all $\cF$ separates points so $\theta$ is injective.
  % To see that $\theta$ is continuous, it is enough to consider a sub-basic open set $U = \pi_f^{-1}(V) \subseteq{ \bC^\cF}$, where $V \subseteq \bC$ is open and $\pi_f$ is the projection on the $f$th coördinate.
  % Then $\theta^{-1}(U) = f^{-1}(V)$ is open.
  % In order to show that $\theta$ is a homeomorphism with its image it will be enough to show that for $F \subseteq X$ closed and $x \notin F$ there is a closed set $F' \subseteq \bC^\cF$ such that $\theta(F) \subseteq F'$ and $\theta(x) \notin F'$.
  % Since $\cF$ separates points from closed sets there is $f \in \cF$ such that $f\rest_F = t$ and $f(x) \neq t$.
  % Then $F' = \{\bar y \in \bC^\cF : y_f = t\}$ will do.
\end{proof}

\begin{dfn}
  \label{dfn:CompletelyRegular}
  Let $X$ be a topometric space.
  Say that a family of functions $\cF \subseteq \cC_{\cL(1)}(X)$ is \emph{sufficient} if
  \begin{enumerate}
  \item It separates points and closed sets.
  \item For $x,y \in X$ we have
    \begin{gather*}
      d(x,y) = \sup \bigl\{ |f(x)-f(y)| : f \in \cF \bigr\}.
    \end{gather*}
    (Clearly, $\geq$ always holds.)
  \end{enumerate}
  A topometric space $X$ is \emph{completely regular} if $\cC_{\cL(1)}(X)$ is sufficient.
\end{dfn}

In view of \fref{fct:TopologicalCompletelyRegularEmbedding} we may say that a topometric space $X$ is completely regular if $\cC_{\cL(1)}(X)$ captures both the topological structure and the metric structure of $X$.

\begin{lem}
  \label{lem:TopometricCompletelyRegularEmbedding}
  Let $X$ be a topometric space space, and let $\cF \subseteq \cC_{\cL(1)}(X)$ be sufficient.
  Then the map $\theta\colon X \to \bC^\cF$ defined by $x \mapsto (f \mapsto f(x))$ is a topometric embedding.
\end{lem}
\begin{proof}
  Immediate.
\end{proof}

\begin{thm}
  \begin{enumerate}
  \item Let $X$ be a completely regular topometric space, and let $M = \diam X \in [0,\infty]$.
    Then $X$ embeds in a power of $[0,M]$.
  \item Every compact or normal topometric space is completely regular.
  \item Every subspace of a completely regular space is completely regular.
  \item Let $X$ be a maximal topometric space.
    Then it is topologically completely regular if and only if it is topometrically completely regular.
  \end{enumerate}
\end{thm}
\begin{proof}
  For the first item, let $\cF_0 \subseteq \cC_{\cL(1)}(X,\bR^+)$ consist of those $f$ such that $\inf f = 0$.
  Then $\cF_0$ is sufficient as well, and consists of functions $f\colon X \rightarrow [0,M]$, so \fref{lem:TopometricCompletelyRegularEmbedding} yields the desired embedding.

  The second item follows from \fref{thm:TopometricUrysohn}, keeping in mind that a compact topometric space is normal, and that since the metric of a topometric space $X$ refines its topology, if $F \subseteq X$ is closed and $x \notin F$ then $d(x,F) > 0$.

  For the third item, assume that $X$ is completely regular and $Y \subseteq X$.
  If $F \subseteq Y$ is closed then $F = Y\cap \overline F$, where $\overline F$ is the closure in $X$.
  Thus if $x \in Y \setminus F$ then $x \in X \setminus \overline F$, so there is a $1$-Lipschitz continuous function separating $\overline F$ from $x$, and its restriction to $Y$ is continuous and $1$-Lipschitz as well.
  A similar restriction argument works for witnessing distances.

  The last item follows from the fact that if $X$ is equipped with the $0/\alpha$ distance then every function to $[0,\alpha]$ is $1$-Lipschitz.
\end{proof}

\begin{cor}
  A topometric space admits a compactification if and only if it is completely regular.
\end{cor}

\begin{thm}
  \label{thm:StoneCechUniversalProperty}
  Let $X$ be completely regular.
  Then it admits a compactification $\beta X$ satisfying the following universal property:
  \begin{quote}
    Every $1$-Lipschitz continuous function $f\colon X \to [0,\infty]$ can be extended to such a function on $\beta X$ (and the extension is unique).
  \end{quote}
  Moreover, $\beta X$ is unique up to a unique isomorphism (i.e., isometric homeomorphism) and satisfies the same universal property with any compact topometric space $Y$ instead of $[0,\infty]$.
\end{thm}
\begin{proof}
  Let $\cF = \cC_{\cL(1)}(X,\bR^+)$ and let $\theta\colon X \to (\bR^+)^\cF \subseteq [0,\infty]^\cF$ be as in \fref{lem:TopometricCompletelyRegularEmbedding}.
  Identify $X$ with $\theta(X)$ and let $\beta X$ be its closure in $[0,\infty]^\cF$.

  For $f \in \cF$, let $\pi_f\colon [0,\infty]^\cF \to [0,\infty]$ be the projection on the $f$th coordinate.
  Then $\pi_f \circ \theta = f$, so $\pi_f\colon \beta X \to [0,\infty]$ is as required.
  Given $f \in \cC_{\cL(1)}(X,[0,\infty])$ and $n \in \bN$, the truncation $f \wedge n\colon X \to [0,n]$ belongs to $\cF$ and the sequence $\pi_{f \wedge n}$ is increasing, converging point-wise to some $g\colon \beta X \to [0,\infty]$.
  The collection of open subsets of $[0,\infty]$ which are either bounded or contain $\infty$ forms a basis.
  For such an open set $U$ there is $n$ such that either $[n,\infty] \subseteq U$ or $U \cap [n,\infty] = \emptyset$, and in either case $g^{-1}(U) = (f \wedge n)^{-1}(U)$ is open.
  Thus $g$ is continuous.
  (Of course we could have also let $\cF = \cC_{\cL(1)}(X,[0,\infty])$ to begin with.)

  Now let $Y$ be any compact topometric space.
  Then $Y$ embeds in $[0,\infty]^J$ for some $J$.
  If $f \in \cC_{\cL(1)}(X,Y)$ then $\pi_j \circ f \in \cC_{\cL(1)}(X,[0,\infty])$ for $j \in J$ and thus extends to $g_j \in \cC_{\cL(1)}(\beta X,[0,\infty])$.
  Let $g = (g_j)\colon \beta X \to [0,\infty]^J$, so $g\rest_X = f$.
  Then $g(X) \subseteq Y$, $X$ is dense in $\beta X$ and $Y$ is closed in $[0,\infty]^J$, so $g(\beta X) \subseteq Y$ as well.

  The uniqueness of an object satisfying this universal property is now standard.
\end{proof}

In other words, for every compact $Y$ the restriction $\cC_{\cL(1)}(\beta X,Y) \to \cC_{\cL(1)}(X,Y)$ is bijective.

\begin{dfn}
  The compactification $\beta X$, if it exists (i.e., if $X$ is completely regular) is called the \emph{Stone-Čech compactification} of $X$.
\end{dfn}

Automorphism groups of metric structures probably form the most natural class of examples of non (locally) compact topometric spaces.
They are easily checked to be completely regular.

\begin{prp}
  Let $\cM$ be a metric structure and let $G = \Aut(\cM)$, equipped with the topology $\sT$ of point-wise convergence and with the distance $d_u$ of uniform convergence.
  Then $(G,\sT,d_u)$ is a completely regular topometric space.

  Similarly, if $(G,\sT)$ is any metrisable topological group, with left-invariant compatible distance $d_L$, and $d_u(f,g) = \sup_h d_L(fh,gh)$, then $(G,\sT,d_u)$ is a completely regular topometric space.
\end{prp}
\begin{proof}
  Since $d_u(f,g) = \sup_{a \in M} d(fa,ga)$, and for each $a$ the function $(f,g) \mapsto d(fa,ga)$ is continuous, $d_u$ is lower semi-continuous.
  Assume that $d_u(f,g) > r$.
  Then there exists $a \in M$ such that $d(fa,ga) > r$, and we may define $\theta(x) = d(fa,xa)$.
  Then $\theta$ is continuous and $1$-Lipschitz (by definition of point-wise and uniform convergence).
  In addition, $\theta(f) = 0$ and $\theta(g) > r$.
  Thus continuous $1$-Lipschitz functions witness distances, and it follows that $d_u$ is lower semi-continuous.
  Now let $U$ be a topological neighbourhood of $f$.
  Then there is a finite tuple $\bar a \in M^n$ and $\varepsilon > 0$ such that $U$ contains the set
  \begin{gather*}
    U_{\bar a,f\bar a,\varepsilon}
    = \{ h : d(h\bar a,f\bar a) < \varepsilon\}.
  \end{gather*}
  Then the function $\rho(x) = d(f\bar a,x\bar a)$ separates $f$ from $G \setminus U$.

  A similar reasoning applies to the case of an abstract group (acting on itself on the left).
  In fact, when $G$ is completely metrisable then this case can be shown to be a special case of the first, and every metrisable group can be embedded in a completely metrisable one.
\end{proof}

\begin{qst}
  Are automorphism groups of metric structures topometrically normal?
  In other words, do continuous $1$-Lipschitz functions witness distance between closed sets?
\end{qst}

\section{Completions}
\label{sec:Completions}

Most topometric spaces one would encounter, such as compact ones (e.g., type spaces) or automorphism groups, are (metrically) complete.
If $X$ is a complete topometric space and $X_0 \subseteq X$ is metrically dense then one can recover the underlying metric space of $X$ from $X_0$, but what about the topology?
Let us first observe that one cannot always expect to be able to recover the topology.

\begin{exm}
  \label{exm:BadCompletion}
  Let $X$ be the disjoint union of $[0,1]$ with $\bN$, where $[0,1]$ is equipped with the usual topology and distance, $\bN$ is equipped with the discrete topology and $0/1$ distance, the distance between any point of $[0,1]$ and of $\bN$ is one, and $0$ (hereafter always referring to $0 \in [0,1]$ and not to $0 \in \bN$) is the limit of $\bN$.
  Thus $X$ is a compact topometric space, which can be naturally viewed as a subspace of $[0,1]^\bN$ by sending $t \in [0,1]$ to $(t,0,0,\ldots)$, and sending $n \in \bN$ to the sequence $(0,0,\ldots,0,1,1,\ldots)$ consisting of $n$ initial zeroes.
  Then $X_0 = (0,1) \cup \bN$ is metrically dense in $X$, but one cannot recover from it the fact that $0$ is an accumulation point of $\bN$.
\end{exm}

Worse still, sometime $\widehat X$ does not even admit a topometric structure extending that of $X$.

\begin{exm}
  \label{exm:NoCompletion}
  Let $X = [0,1]^2$, equipped with the usual topology.
  We put the usual distance on $(0,1] \times \{0\}$, and make any two points not both there have distance one.
  It is clear that $d$ refines the topology, and a case-by-case consideration yields that if $d(x,y) > r$ then there are neighbourhoods $x \in U$, $y \in V$ such that $d(U,V) > r$, so $d$ is lower semi-continuous and $X$ is a topometric space.

  Now, $\widehat X \setminus X$ consists of a single new point $*$, the metric limit of $(t,0)$ as $t \rightarrow 0$.
  If $\widehat \sT$ is any topology on $\widehat X$ coarser then $d_{\widehat X}$ and agreeing with the given topology on $X$, then $*$ and $(0,0)$ cannot be separated by open sets, so $\widehat X$ cannot even be topologically Hausdorff, let alone a topometric space.
\end{exm}

Thus two questions arise: which topometric spaces admit topometric completions, and under which conditions we can recover the topology on a complete topometric space $X$ from the restriction to a metrically dense $X_0 \subseteq X$.

\begin{prp}
  \label{prp:CompletionCriterion}
  Let $X$ be a topometric space.
  Then the following are equivalent:
  \begin{enumerate}
  \item \label{item:CompletionCriterion}
    For every $x,y \in \widehat X$ there exist $\varepsilon > 0$ and open sets $U,V \subseteq X$ such that $B(x,\varepsilon) \cap X \subseteq U$, $B(y,\varepsilon) \cap X \subseteq V$ and $d(U,V) > r$.
  \item The metric completion $\widehat X$ admits a strongest topology $\widehat \sT$ rendering it a topometric extension of $X$.
    A set $U \subseteq \widehat X$ belongs to $\widehat \sT$ if and only if it is metrically open in $\widehat X$ and $U \cap X$ is open in $X$.

    Moreover, for $V \subseteq X$ open let $V' \subseteq \widehat X$ consist of all $x$ such that $B(x,\varepsilon) \cap X \subseteq V$ for some $\varepsilon > 0$.
    Then $V' \in \widehat \sT$ is maximal such that $V' \cap X = V$.
  \item There exists an embedding of $X$ into a complete topometric space.
  \end{enumerate}
\end{prp}
\begin{proof}
  \begin{cycprf}
  \item
    Let $Y \supseteq X$ be a complete extension of $X$, and let $x,y \in \widehat X \subseteq Y$.
    Since the distance is lower semi-continuous, there are open sets $U,V \subseteq Y$ such that $x \in U$, $y \in V$ and $d(U,V) > r$.
    Since the distance refines the topology, there exists $\varepsilon > 0$ such that $B_Y(x,\varepsilon) \subseteq U$ and $B_Y(y,\varepsilon) \subseteq V$.
    Intersecting with $X$, we are done.
  \item
    Clearly, $\widehat \sT$ is a topology, coarser than $d_{\widehat X}$.
    One easily checks the moreover part, so $(X,\sT) \subseteq (\widehat X,\widehat \sT)$ is a topological embedding.

    Now let $x,y \in \widehat X$, $d(x,y) > r$, and let $U$, $V$ and $\varepsilon$ be as given by our hypothesis.
    Then $x \in U'$ and $y \in V'$.
    Moreover, $d(U',V') = d(U,v) > r$, showing that $d_{\widehat X}$ is lower $\widehat \sT$-semi-continuous, and $(\widehat X,\widehat \sT,d_{\widehat X})$ is a topometric space.

    Finally, let $\widehat \sT'$ be any topology on $\widehat X$ rendering it a topometric extension of $X$.
    In particular, $d_{\widehat X}$ refines $\widehat \sT'$ so $\widehat \sT' \subseteq \widehat \sT$.
  \item[\impfirst]
    Immediate.
  \end{cycprf}
\end{proof}

This answers one question, and allows us to restate the other as, under which conditions is a complete topometric space $X$ the strongest completion of a metrically dense $X_0 \subseteq X$.

\begin{dfn}
  Call a topometric space \emph{completable} if it satisfies the equivalent conditions of \fref{prp:CompletionCriterion}.
\end{dfn}

\begin{prp}
  Every completely regular topometric space $X$ is completable.
  Moreover, the Stone-Čech compactification $\beta X$ contains a (unique) metric copy of $\widehat X$ equipped with strongest completion topology, so the strongest completion is completely regular as well.
\end{prp}
\begin{proof}
  Since $\beta X$ is compact, it is complete, so it contains a copy of $\widehat X$.

  Now let $U \in \widehat \sT$, and let $x \in U \subseteq \widehat X$.
  Then there is some $\varepsilon > 0$ such that $B(x,2\varepsilon) \subseteq U$ and $U \cap X$ is open in $X$.
  Let $y \in B(x,\varepsilon) \cap X$.
  Then by complete regularity there exists $f \in \cC_{\cL(1)}(X,[0,\varepsilon]) = \cC_{\cL(1)}(\beta X,[0,\varepsilon])$ such that $f(y) = \varepsilon$ and $f\rest_{X \setminus U} = 0$.
  Then $x \in \{z \in \widehat X\colon f(z) > 0\} \subseteq U$, so $U$ is open in the topology induced from $\beta X$.
  Since $\widehat \sT$ is strongest possible topology on the completion, it is the induced topology.
\end{proof}

This, however, does not solve the problem raised in \fref{exm:BadCompletion}, since there all spaces were completely regular.
We content ourselves here with providing sufficient conditions for a topometric space to be the unique strongest completion of any metrically dense subset.
Given the rest of this paper, we allow ourselves to assume that $X$ is completely regular, and show that the extra conditions are relatively well behaved under various constructions.

\begin{itemize}
\item[$(*)$]
  For every open set $U \subseteq X$ and $r > 0$, the open metric neighbourhood $B(U,r)$ is (topologically) open.
\item[$(**)$]
  For every open set $U \subseteq X$, $x \in \widehat U \subseteq \widehat X$ and $r > 0$, there exists $\varepsilon > 0$ such that $B(x,\varepsilon) \cap X \subseteq B(U,r)^\circ$.
  Equivalently, there is $V \subseteq B(U,r)$ open such that $\widehat U \subseteq V'$, with $V'$ as in \fref{prp:CompletionCriterion}.
  When $X$ is complete, this is the same as saying that $\widehat U \subseteq B(U,r)^\circ$.
\end{itemize}
Clearly $(*)$ implies $(**)$.

\begin{prp}
  \label{prp:ContinuousLipschitzDenseExtension}
  Let $X$ be a complete, completely regular topometric space and $X_0 \subseteq X$ metrically dense.
  Then $X$ satisfies $(*)$ (respectively, $(**)$) if and only if $X_0$ does and $X$ carries the strongest possible topology as a completion of $X_0$.
\end{prp}
\begin{proof}
  Let $U$ be open in $X_0$, and let $V \subseteq X$ be open such that $U = V \cap X_0$.
  Since the metric refines the topology, $U$ is $d$-dense in $V$, so $B_X(U,r) = B_X(V,r)$.
  When $X$ satisfies $(*)$, this latter set is open, and therefore so is $B_{X_0}(U,r) = B_X(V,r) \cap X_0$, whence $(*)$ for $X_0$.
  Assuming that $X$ satisfies $(**)$, there is $W \subseteq X$ open such that $\widehat V \subseteq W \subseteq B_X(V,r)$, in which case $W \cap X_0 \subseteq B_{X_0}(U,r)$ and $\widehat U = \widehat V \subseteq W \subseteq (W \cap X_0)'$, and we have $(**)$ for $X_0$.

  Conversely, let us assume that $X$ carries the strongest completion topology, and let $V \subseteq X$ be open, $U = V \cap X_0$.
  If $X_0$ satisfies $(*)$ then $B_X(V,r) \cap X_0 = B_{X_0}(U,r)$ is open in $X_0$, and then $B(V,r)$ is open in $X$.
  If $X_0$ only satisfies $(**)$, there exists $W \subseteq X$ open such that $W \cap X_0 \subseteq B_{X_0}(U,r) = B_X(V,r) \cap X_0$ and $\widehat V = \widehat U \subseteq (W \cap X_0)'$.
  But then $(W \cap X_0)' \subseteq B_X(V,r)$ as well, so $X$ has $(**)$.

  Assume now that $X$ satisfies $(**)$ and let us show that it is the strongest completion of $X_0$.
  Indeed, all we need to show is that if $f \in \cC_{\cL(1)}(X_0)$ then $\hat f$, the unique $1$-Lipschitz extension of $f$ to $X$, is continuous at every $x \in X$.
  Assuming, as we may, that $\hat f(x) = 0$, let $U = \{y \in X_0 : |f(y)| < \varepsilon\}$ for some $\varepsilon > 0$.
  Then $U \subseteq X_0$ is open, so of the form $V \cap X$ for some open $V \subseteq X$, and $x \in \widehat U = \widehat V \subseteq B_X(V,\varepsilon)^\circ \subseteq \{y \in X : | \hat f(y)| < 2\varepsilon\}$.
  Since $\varepsilon$ was arbitrary, $\hat f$ is continuous at $x$.
\end{proof}

\begin{lem}
  Condition $(*)$ holds in every topometric space of the form $\prod [s_i,r_i]$.
  More generally, it holds in every minimal or maximal topometric space, and if it holds in each $X_i$ then it holds in $\prod X_i$.
  Similarly, if condition $(**)$ holds in each $X_i$ then it also holds in $\prod X_i$.
\end{lem}
\begin{proof}
  Easy.
\end{proof}

\begin{lem}
  Condition $(*)$ holds in every topometric group.
  In fact, while we usually require that the distance in a topometric group be biïnvariant, here it is enough that it be invariant on one side.
\end{lem}
\begin{proof}
  Assume that the distance is left-invariant.
  Then one checks that $B(U,r) = \bigcup_{d(h,1) < r} Uh$.
\end{proof}

In addition, type spaces in continuous logic, equipped with the usual distance are compact, so completely regular, and satisfy $(*)$, and the same holds for local type spaces (i.e., spaces of $\varphi$-types for some formula $\varphi$).
On the other hands, there exist compact topometric spaces where the properties discussed in this section fail, e.g., the one given in \fref{exm:BadCompletion} above, as well as of type spaces with ``exotic'' distances:

\begin{exm}
  In \cite[Example~3.11 \& Theorem~3.15]{BenYaacov:Perturbations} an example was given somewhat indirectly of a compact topometric space in which condition $(*)$ fails (in the terminology used there, the perturbation distance is not \emph{open} or even weakly so).
\end{exm}

\section{An abstract characterisation of continuous $1$-Lipschitz functions}
\label{sec:Abstract1Lipschitz}

A compact Hausdorff topological space $X$ can be recovered, up to a unique homeomorphism, either from $\cC(X)$ viewed as a $C^*$-algebra, or from $\cC(X,\bR)$ viewed as a Banach lattice (see \fref{dfn:BanachLattice} and \fref{fct:Kakutani} below).
Accordingly, we wish to recover a compact topometric space from its space of continuous $1$-Lipschitz functions.
While $\cC_{\cL(1)}(X)$ is not an algebra, $\cC_{\cL(1)}(X,\bR)$ is a lattice, so it is the real version which we prefer to use as a base.
Of course, the lattice $\cC(X,\bR)$ can be recovered from the $C^*$-algebra $\cC(X)$ as the space of self-adjoint elements together with continuous functional calculus applied to $\min$ and $\max$, so there is no real difference to which case we start with.

\begin{lem}[Lattice Stone-Weierstraß for Lipschitz functions]
  \label{lem:LatticeStonrWeierstrassLipschitz}
  Let $X$ be a compact topometric space, $S \subseteq \bR$, and let $A \subseteq \cC_{\cL(1)}(X,S)$ be a sub-lattice, such that for every $x,y \in X$ (possibly equal), $s,t \in S$ and $\varepsilon > 0$ such that $|s-t| \leq d(x,y)$ there is some $f \in A$ with $| f(x) - s | + | f(y) - t | < \varepsilon$.
  Then $A$ is dense in $\cC_{\cL(1)}(X,S)$ in uniform convergence.
\end{lem}
\begin{proof}
  The proof is very similar to the classical argument.
  Indeed, let $h \in \cC_{\cL(1)}(X,\bR)$ and $\varepsilon > 0$.
  By assumption, for every $x,y \in X$ there is $f_{xy} \in A$ such that $| f_{xy}(x) - h(x) | + | f_{xy}(y) - h(y) | < \varepsilon$.
  In particular, $f_{xy}(x) < h(x) + \varepsilon$ and $f_{xy}(y) > h(x) - \varepsilon$.
  Fixing $x$, for each $y$ there is a neighbourhood $V_y$ on which $f_{xy} > h - \varepsilon$, and by compactness there is a finite family $\{y_i\}_{i<n}$ such that $\bigcup_i V_{y_i} = X$.
  Letting $g_x = \bigvee_i f_{x y_i} \in A$ we have $g_x(x) < h(x) + \varepsilon$ and $g_x > h - \varepsilon$ everywhere.
  By a similar argument, there is a finite family $\{x_j\}_{j < m}$ such that $h' = \bigwedge_j g_{x_i} \in A$ satisfies $h - \varepsilon < h' < h + \varepsilon$, as desired.
\end{proof}

\begin{conv}
  When $E$ is a Banach space and also a lattice, by a \emph{convex sub-lattice} $A \subseteq E$ we mean a subset which is convex with respect to the linear structure, and closed with respect to the lattice operations.
  In other words, we interpret the words ``convex'' and ``sub-lattice'' separately from the rest of the structure, so we do not mean that $A$ is a Banach space, nor do we mean that if $f,g \in A$, $h \in E$ and $f \leq h \leq g$ in the lattice order then $h \in A$.
  We say that $A$ is \emph{symmetric} if $A = -A$.
\end{conv}

\begin{thm}
  \label{thm:L1Set}
  Let $X$ be a compact topological space, $A \subseteq \cC(X,\bR)$.
  Then the following are equivalent:
  \begin{enumerate}
  \item There is a topometric structure $(X,d)$ on $X$ such that $A = \cC_{\cL(1)}(X,\bR)$.
  \item The following hold:
    \begin{itemize}
    \item The set $A$ is closed (in norm).
    \item The set $A$ is a symmetric convex sub-lattice of $\cC(X,\bR)$ containing $\bR$.
    \item The set $A$ separates point of $X$.
    \end{itemize}
  \item As above, with the second point replaced with
    \begin{itemize}
    \item The set $A$ is a sub-lattice, closed under translation by $\alpha \in \bR$ and multiplication by $\alpha \in [-1,1]$.
    \end{itemize}
  \end{enumerate}
  In this case $A$ is weakly closed, generates $\cC(X,\bR)$ as a Banach space, and the metric $d$ is unique and can be recovered by
  \begin{gather}
    \label{eq:DistanceFrom1Lipschitz}
    d(x,y) = \sup_{f \in A} |f(x)-f(y)|.
  \end{gather}
\end{thm}
\begin{proof}
  \begin{cycprf*}
  \item
    Easy.
  \item
    Since $0 \in A$ and $A$ is convex, it is closed under multiplication by $\alpha \in [0,1]$, and since $A$ is symmetric, it is closed under multiplication by $\alpha \in [-1,1]$.
    Now let $\alpha \in \bR$ and $f \in A$, and for $t \in [0,1)$ let $f_t = tf + (1-t) \frac{\alpha}{1-t} \in A$.
    Then $f_t \rightarrow f + \alpha$ so $f + \alpha \in A$ as well.
  \item[\impfirst]
    Let us define $d$ by \fref{eq:DistanceFrom1Lipschitz}.
    Since $A$ generates $\cC(X)$, $d$ is a distance, and it is lower semi-continuous as the supremum of continuous functions.
    It follows that for each $x \in X$, the closed balls $\overline B(x,r)$ are closed.
    Thus, if $F$ is a closed set not containing $x$, then by compactness there is some $r > 0$ such that $F \cap B(x,r) = \emptyset$, so $d$ refines the topology.
    Thus $(X,d)$ is a topometric space, and we view it henceforth as such.

    It is immediate from the construction that $A \subseteq \cC_{\cL(1)}(X,\bR)$.
    Now assume that $x,y \in X$, $0 \leq s-t \leq d(x,y)$ and $\varepsilon > 0$.
    By definition of $d$ there is $f \in A$ such that $s-t < |f(x) - f(y)| + \varepsilon$, and possibly multiplying $f$ by some $\alpha \in [-1,1]$ we may assume that $s-t < f(x) - f(y) + \varepsilon < s-t + \varepsilon$.
    Up to translation, we may further have $t = f(y)$, in which case $s < f(x) < s + \varepsilon$.
    By \fref{lem:LatticeStonrWeierstrassLipschitz}, $A = \cC_{\cL(1)}(X,\bR)$.
  \end{cycprf*}
  The identity \fref{eq:DistanceFrom1Lipschitz} follows from Urysohn's Lemma for normal topometric spaces and the fact that a compact topometric space is normal.
  The set $\bigcup_n nA = \cC_\cL(X,\bR)$ is dense in $\cC(X,\bR)$ by \fref{cor:DenseLipschitz}, and it is clear that $\cC_{\cL(1)}(X,\bR)$ is closed in point-wise convergence, which, in $\cC(X)$, coincides with the weak topology.
\end{proof}

This is quite different from \cite[Theorem~4.3.2]{Weaver:LipschitzAlgebras}, which still seems to be the most closely analogous result therein.

If one desires a characterisation of $\cC_{\cL(1)}(X,\bR)$ which does not make any reference to an ambient $\cC(X,\bR)$ or $\cC(X)$, one first requires an abstract characterisation of symmetric convex sub-lattices of Banach lattices.
Convex spaces, i.e., convex subsets of Banach spaces, equipped with the induced metric and convex combination operations $c_t(x,y) = tx + (1-t)y$ for $t \in [0,1]$, are characterised, for example, by Machado \cite{Machado:Convex} (see also \cite{BenYaacov:NakanoSpaces} for a characterisation with a slightly different set of axioms using the convex combination $\half[x+y]$ alone).
Such a space is called \emph{symmetric} when they are equipped with a constant $0$ and a unary operation $-x$ such that $c_{1/2}(x,-x) = 0$ for all $x$.
It then embeds isometrically as a convex symmetric subset of a Banach space, and the generated Banach space is unique up to a unique isomorphism.

The next step is to characterise when a lattice structure on a symmetric convex space comes from a Banach lattice structure on the generated Banach space.
First, let us recall the following from Meyer-Nieberg \cite{MeyerNieberg:BanachLattices}.

\begin{dfn}
  \label{dfn:BanachLattice}
  \begin{enumerate}
  \item An \emph{ordered vector space} is a vector space $(E,\leq)$ over an ordered field $(k,\leq)$ satisfying for all $v,u,w \in E$ and $0 < \alpha \in k$: $v \leq u \Longleftrightarrow v+w \leq u + w \Longleftrightarrow \alpha v \leq \alpha u$.
  \item An ordered vector space is a \emph{vector lattice} (or a \emph{Riesz space}) if it is a lattice, i.e., if every two $v,u \in E$ admit a least upper bound (or \emph{join}) $v\vee u$ and a greatest lower bound (or \emph{meet}) $v\wedge u$.
    In this case we write $|v| = v \vee (-v)$, $v^+ = v \vee 0$, $v^- = (-v) \vee 0$.
  \item \label{item:NormLat}
    A vector lattice over $\bR$ is a \emph{normed vector lattice} it admits a norm satisfying: $|v| \leq |u| \Longrightarrow \|v\| \leq \|u\|$.
  \item A \emph{Banach lattice} is a complete normed vector lattice.
  \end{enumerate}
\end{dfn}

For our purposes, a list of identities will be better.

\begin{lem}
  \label{lem:BanachLattice}
  Let $E$ be a vector lattice.
  Then the following hold:
  \newcounter{enumkeep}
  \begin{enumerate}
  \item \label{item:Lattice}
    $(E,\vee,\wedge)$ is a distributive lattice, i.e., the operations $\vee$ and $\wedge$ are idempotent, commutative, associative, associative over one another, and satisfy the absorption axiom $(v\wedge u)\vee v = (v\vee u)\wedge v = v$.
  \item \label{item:LatticeAddition}
    $(v + w) \vee (u + w) = (v \vee u) + w$.
  \item \label{item:LatticeScaling}
    For scalar $\alpha > 0$: $(\alpha v) \vee (\alpha u) = \alpha(v \vee u)$.
  \item \label{item:LatticeExcludedMiddle}
    $v + u = v \vee u + v \wedge u$.
    \setcounter{enumkeep}{\value{enumi}}
  \end{enumerate}
  If it is a normed vector lattice then it satisfies in addition:
  \begin{enumerate}
    \setcounter{enumi}{\value{enumkeep}}
  \item \label{item:LatticeNorm}
    $\| v \| = \big\| |v| \big\| \leq \bigl\| |v| \vee u \bigr\|$.
  \end{enumerate}

  Conversely, assume $E$ is a vector space equipped in addition with operations $\vee$,$\wedge$ verifying \fref{item:Lattice}-\fref{item:LatticeExcludedMiddle} then $E$ is a vector lattice where the order can be recovered as $v \leq u \Longleftrightarrow v\wedge u = v \Longleftrightarrow v \vee u = u$.

  If it is a normed (Banach) space verifying \fref{item:Lattice}-\fref{item:LatticeNorm} then it is a normed (Banach) lattice.
\end{lem}
\begin{proof}
  The first statements are \cite[Theorem~1.1.1]{MeyerNieberg:BanachLattices} and easy verifications.

  For the first part of the converse we observe that $v \leq u \Longleftrightarrow v\wedge u = v \Longleftrightarrow v \vee u = u$ does indeed define a partial order with respect to which $\vee$ and $\wedge$ are the join and meet, respectively.
  From \fref{item:LatticeAddition},\fref{item:LatticeScaling} it follows that $u \leq v \Longleftrightarrow 0 \leq v-u \Longleftrightarrow 0 \leq \alpha(v-u)$ for $\alpha > 0$, whereby $E$ is a vector lattice.

  For the second part of the converse, if $|v| \leq |u|$ then $\|v\| = \bigl\| |v| \bigr\| \leq \bigl\| |v| \vee |u| \bigr\| = \bigl\| |u| \bigr\| = \|u\|$.
\end{proof}

\begin{lem}
  \label{lem:ConvexSymmetricLattice}
  A complete metric space $A$ together with operations $-,c_t,\vee,\wedge$ (for $t \in [0,1]$) is a symmetric convex sub-lattice of a Banach lattice if and only if:
  \begin{enumerate}
  \item $(A,0,-,c_t)_{t \in [0,1]}$ is a symmetric convex space.
  \item $(A,,\vee,\wedge)$ is a distributive lattice.
  \item For all $u,v,w \in A$ and $t \in [0,1]$:
    \begin{gather*}
      c_t(v,w) \vee c_t(u,w) = c_t(v \vee u, w), \qquad
      c_{1/2}(v \vee u,v \wedge u) = c_{1/2}(v,u), \\
      \| v \| = \big\| |v| \big\| \leq \bigl\| |v| \vee u \bigr\|,
    \end{gather*}
    where $|v| = v \vee (-v)$ as before.
  \end{enumerate}
  Moreover, in this case the Banach lattice can be taken to be the Banach space generated by $A$, on which the lattice structure is uniquely determined.
\end{lem}
\begin{proof}
  Let $E_0$ be the normed linear space generated by $A$, whose elements are all of the form $\alpha v$ for some $\alpha \geq 1$ and $v \in A$, and let $E$ be its completion.
  Substituting $w = 0$ above we obtain $(\alpha v) \vee (\alpha u) = \alpha(v \vee u)$ for $0 \leq \alpha \leq 1$, so we may unambiguously define $v \vee u = \bigl( (v/\alpha) \vee (u/\alpha) \bigr) \alpha$ for any $v,u \in E_0$ and $\alpha > 0$ big enough so that $v/\alpha, u/\alpha \in A$.

  As in the proof of \fref{lem:BanachLattice}, the lattice structure induces an order on $A$ for which $\vee$ and $\wedge$ are the join and meet.
  We also have $v \leq u$ if and only if $c_{1/2}(u,-v) \geq 0$ if and only if $-u \leq -v$, whereby $v \wedge u = -\bigl( (-v) \vee (-u) \bigr)$ on $A$.
  We may then extend $\wedge$ from $A$ to $E_0$ by scaling as well.
  The hypotheses of \fref{lem:BanachLattice} hold in $E_0$ by scaling to $A$, so $E_0$ is a normed lattice, and $E$ carries a unique Banach lattice structure.
\end{proof}

Banach lattices of the form $\cC(X,\bR)$ are characterised as follows.

\begin{fct}[{Kakutani, see Meyer-Nieberg \cite[Theorem~2.1.3]{MeyerNieberg:BanachLattices}}]
  \label{fct:Kakutani}
  Let $E$ be a Banach lattice containing a greatest element of norm at most one, which we call $1$, and let $X = X(E) \subseteq E'$ be the set of extreme positive linear functional of norm one, equipped with the weak$^*$ topology.
  Then $X$ is compact, and $E \cong C(X,\bR)$, with $1$ being sent to the constant function $1$.

  Conversely if $X$ is a compact space, then $E = \cC(X,\bR)$ is as above, and $X(E)$ is canonically homeomorphic to $X$.
\end{fct}

And we conclude:

\begin{thm}
  \label{thm:CompactTopometricLipschitzSpace}
  Let $X$ be a compact topometric space, $A = (A,0,1,-,c_t,\vee,\wedge)_{t \in [0,1]} = \cC_{\cL(1)}(X,\bR)$.
  Then $(A,0,-,c_t,\vee,\wedge)_{t \in [0,1]}$ is a convex symmetric lattice, i.e., it satisfies the hypotheses of \fref{lem:ConvexSymmetricLattice}, and for every $v \in A$ we have
  \begin{gather*}
    \|v\| = \sup \, \bigl\{ (1-t)/t : 0 < t \leq 1, \, c_t(|v|,-1) \leq 0 \bigr\},
  \end{gather*}
  where $\|v\| = d(v,0)$ and $u \leq 0$ means $u \vee 0 = 0$.

  Conversely, every $(A,0,1,-,c_t,\vee,\wedge)_{t \in [0,1]}$ satisfying these properties is isomorphic to $\cC_{\cL(1)}(X,\bR)$ for some compact topometric space $X$, which is moreover unique up to a unique isomorphism.
  Moreover, let $C \subseteq \bR^A$ be the set of all functions $\lambda$ such that $\lambda c_t(u,v) = t\lambda u + (1-\lambda)v$, $\lambda(-v) = -\lambda v$, $\lambda |v| \geq 0$ and $\lambda 1 = 1$.
  Then $C$ is convex and compact, $X$ can be identified with the set of extreme points of $C$, and for $\lambda,\mu \in X$ we have
  \begin{gather*}
    d(\lambda,\mu) = \sup_{v \in A} |\lambda v - \mu v|.
  \end{gather*}
\end{thm}
\begin{proof}
  The main assertion is clear.
  For the converse, we first embed $A$ into the generated Banach lattice $E$ as per \fref{lem:ConvexSymmetricLattice}.
  Then $(E,1)$ satisfies the hypotheses of \fref{fct:Kakutani}, allowing us to recover the compact topological space $X$.
  The topometric structure on $X$ is then recovered by \fref{thm:L1Set}.
  For the moreover part, construct the dual $E^*$ directly as a space of functions on $A$, then unwind the other constructions (for a positive functional $\lambda$ we have $\|\lambda\| = \lambda 1$, so continuity of $\lambda$ comes for free).
\end{proof}

\bibliographystyle{begnac}
\bibliography{begnac}

\providecommand{\bysame}{\leavevmode\hbox to3em{\hrulefill}\thinspace}
\begin{thebibliography}{{Ben}08b}

\bibitem[BBM]{BenYaacov-Berenstein-Melleray:TopometricGroups}
Itaï \bgroup\scshape{}{Ben Yaacov}\egroup{}, Alexander
  \bgroup\scshape{}Berenstein\egroup{}, and Julien
  \bgroup\scshape{}Melleray\egroup{},
  \href{http://math.univ-lyon1.fr/~begnac/articles/TopMetGen.pdf} {\emph{Polish
  topometric groups}}, Transactions of the American Mathematical Society, to
  appear, \href{http://arxiv.org/abs/1007.3367}{arXiv:1007.3367}.

\bibitem[{Ben}08a]{BenYaacov:Perturbations}
Itaï \bgroup\scshape{}{Ben Yaacov}\egroup{},
  \href{http://math.univ-lyon1.fr/~begnac/articles/Perturb.pdf} {\emph{On
  perturbations of continuous structures}}, Journal of Mathematical Logic
  \textbf{8} (2008), no.~2, 225--249,
  \href{http://dx.doi.org/10.1142/S0219061308000762}{doi:10.1142/S0219061308000762},
  \href{http://arxiv.org/abs/0802.4388}{arXiv:0802.4388}.

\bibitem[{Ben}08b]{BenYaacov:TopometricSpacesAndPerturbations}
\bysame, \href{http://math.univ-lyon1.fr/~begnac/articles/TopoPert.pdf}
  {\emph{Topometric spaces and perturbations of metric structures}}, Logic and
  Analysis \textbf{1} (2008), no.~3--4, 235--272,
  \href{http://dx.doi.org/10.1007/s11813-008-0009-x}{doi:10.1007/s11813-008-0009-x},
  \href{http://arxiv.org/abs/0802.4458}{arXiv:0802.4458}.

\bibitem[{Ben}09]{BenYaacov:NakanoSpaces}
\bysame, \href{http://math.univ-lyon1.fr/~begnac/articles/Nakano.pdf}
  {\emph{Modular functionals and perturbations of {N}akano spaces}}, Journal of
  Logic and Analysis \textbf{1:1} (2009), 1--42,
  \href{http://dx.doi.org/10.4115/jla.2009.1.1}{doi:10.4115/jla.2009.1.1},
  \href{http://arxiv.org/abs/0802.4285}{arXiv:0802.4285}.

\bibitem[BU10]{BenYaacov-Usvyatsov:CFO}
Itaï \bgroup\scshape{}{Ben Yaacov}\egroup{} and Alexander
  \bgroup\scshape{}Usvyatsov\egroup{},
  \href{http://math.univ-lyon1.fr/~begnac/articles/cfo.pdf} {\emph{Continuous
  first order logic and local stability}}, Transactions of the American
  Mathematical Society \textbf{362} (2010), no.~10, 5213--5259,
  \href{http://dx.doi.org/10.1090/S0002-9947-10-04837-3}{doi:10.1090/S0002-9947-10-04837-3},
  \href{http://arxiv.org/abs/0801.4303}{arXiv:0801.4303}.

\bibitem[Mac73]{Machado:Convex}
Hilton~Vieira \bgroup\scshape{}Machado\egroup{}, \emph{A characterization of
  convex subsets of normed spaces}, K\=odai Mathematical Seminar Reports
  \textbf{25} (1973), 307--320.

\bibitem[{Mey}91]{MeyerNieberg:BanachLattices}
Peter \bgroup\scshape{}{Meyer-Nieberg}\egroup{}, \emph{Banach lattices},
  Universitext, Springer-Verlag, Berlin, 1991.

\bibitem[Mun75]{Munkres:Topology}
James~R. \bgroup\scshape{}Munkres\egroup{}, \emph{Topology: a first course},
  Prentice-Hall Inc., Englewood Cliffs, N.J., 1975.

\bibitem[Wea99]{Weaver:LipschitzAlgebras}
Nik \bgroup\scshape{}Weaver\egroup{}, \emph{Lipschitz algebras}, World
  Scientific Publishing Co. Inc., River Edge, NJ, 1999.

\end{thebibliography}

\end{document}